\documentclass[12pt]{article}
\usepackage{cite}
\usepackage{amsmath,amsthm,amssymb}
\usepackage[left=1.5in,top=1in,right=1in,nohead]{geometry}
\usepackage{hyperref}
\usepackage{color}
\usepackage{latexsym}
\usepackage[sc,osf]{mathpazo}
\usepackage{cite}
\usepackage{amsmath,amssymb,amsthm}
\title{Semi-invariant Submanifolds of Normal Complex Contact Metric Manifolds  }
\author{\textbf{Aysel TURGUT VANLI}\\
	{\small Department of Mathematics, 	Gazi University,
	TURKEY}\\{\small avanli@gazi.edu.tr}\\
	\textbf{\.Inan \"Unal}\\{\small 
	Department of Computer Engineering, 	Munzur University,
	TURKEY}\\{\small inanunal@munzur.edu.tr}}

\theoremstyle{plain}
\newtheorem{defn}{Definition}

\theoremstyle{plain}

\newtheorem{thm}{Theorem}[section]
\newtheorem{cor}[thm]{Corollary}
\newtheorem{lem}[thm]{Lemma}
\newtheorem{prop}[thm]{Proposition}
\theoremstyle{definition}
\theoremstyle{remark}

\numberwithin{equation}{section}
\usepackage{csquotes}
\date{}

\begin{document}

\maketitle
\setlength{\parindent}{0pt}

\begin{abstract}
	 In this paper, we study on semi-invariant submanifolds of normal complex contact metric manifolds. We give the definition of such submanifolds and we obtain useful relations. Moreover, we give the integrability conditions of distributions. 
\end{abstract}

\maketitle
\section{Introduction}
The complex contact manifold notion is  a branch of contact geometry which has not much studied as much as real contact manifolds. In 2000, Korkmaz \cite{KB2000} gave a normality notion for such manifolds. In \cite{BL2010}, Blair presented a comprehensive introduction to complex contact manifolds. \par 
Submanifold theory of contact manifolds is an interesting topic in contact geometry. Especially, the classes of submanifolds such as, invariant, semi-invariant, anti-invariant, slant etc., have become interesting for researchers, recent years.  In real contact geometry this subject have been studied since 1970s. In 1980s Bejancu \cite{bejancualmost} gave the definition of almost semi-invariant submanifolds and some subclasses. Also same author worked on semi-invariant submanfiolds of Sasakian manifolds \cite{bejancusemi}. Besides, in complex contact geometry the special subclasses of submanifolds have not studied yet. 
This is an area of  awaiting attention, with many open problems. One of presented authors studied on this subject \cite {vanli2019remark,vanli2020comment}. \par

With above reasons  our aim, is to give an introduction for the special submanifolds of complex contact manifolds. We take into consider the normality notion is given by Korkmaz \cite{KB2000}. By this way,  our paper is organized as follow. The first section is on fundamental facts on complex contact manifolds. In the second section we give the definition for a semi-invariant submanifold of a normal complex contact metric manifold and obtain some relations. Finally, we give the integrability conditions of distributions in the last section .
\section{Preliminaries}
The definition of a complex contact manifold was given by Kobayashi \cite{KOb59} as follows:
Let $M$ be a complex manifold of odd complex dimension $2m+1$ covered by an
open covering $\mathcal{A=}\left\{ \mathcal{U}_{i}\right\} $ consisting of
coordinate neighborhoods. If there is a holomorphic $1$-form $\omega _{i}$
on each $\mathcal{U}_{i}\in \mathcal{A}$ in such a way that for any $%
\mathcal{U}_{i},\mathcal{U}_{j}\in \mathcal{A}$ and for a holomorphic function $f_{ij}$ on $\mathcal{U}_{i}\cap \mathcal{U}_{j}\neq \emptyset $
\begin{eqnarray*}
	\omega _{i}\wedge (d\omega _{i})^{m}\neq 0\text{ in }\mathcal{U}_{i}\text{, }\\
	\omega _{i}=f_{ij}\omega _{j},\text{ }\mathcal{U}_{i}\cap \mathcal{U}%
	_{j}\neq \emptyset 
\end{eqnarray*}%
then the set $\left\{ \left( \omega _{i},\mathcal{U}%
_{i}\right) \mid \mathcal{U}_{i}\in \mathcal{A}\right\} $ of local structures is called complex contact structure, and with this structure $ M $ is called a complex contact manifold. 
Let $ (M, \omega_{i}) $ be a complex contact manifold. For every $ p \in M $ we have a subspace of $ T_{p}M $ by kernel of $ \omega_{i} $: 
\begin{equation*}
\mathcal{H}_{i}=\{X_{p}:\omega_{i}(X_{p})=0, X_{p}\in T_{p}M\}.
\end{equation*}
Then on $ \mathcal{U}_{i}\cap \mathcal{U}_{j} \neq0 $ we have $  \mathcal{H}_{i}= \mathcal{H}_{j} $ and so $ \mathcal{H}=\cup \mathcal{H}_{i} $. $  \mathcal{H} $ is well-defined, $ 2m-$complex dimensional non-integrable subbundle on $ M $ and it is called the horizontal subbundle. \\
Let $ M $ be a complex contact manifold of odd complex dimension $ 2m+1 $. Ishihara and Konishi \cite{IK80} proved that $ M $ admits always an almost contact structure of $ C^{\infty} $. They also give the Hermitian metric. An odd complex $ 2m+1 -$dimensional complex manifold with Hermitian metric and almost contact structure is called \textit{complex almost contact metric manifold}. 
\begin{defn} \label{complexalmostscontact}
	Let $M$ be a odd complex $ 2m+1 -$dimensional complex manifold with complex structure $ J $, Hermitian metric $ g $, and $\mathcal{A=}\left\{ 
	\mathcal{U}_{i}\right\} $ be an open covering of $M$ with coordinate
	neighborhoods $\{\mathcal{U}_{i}\mathcal{\}}.$ If $M$ satisfies the
	following two conditions then it is called a \textit{complex almost contact
		metric manifold}:
	
	1. In each $\mathcal{U}_{i}$ there exist $1$-forms $u_{i}$ and $%
	v_{i}=u_{i}\circ J$, with dual vector fields $U_{i}$ and $V_{i}=-JU_{i}$ and 
	$(1,1)$ tensor fields $G_{i}$ and $H_{i}=G_{i}J$ such that 
	\begin{equation} \label{G^2veH^2}
	H_{i}^{2}=G_{i}^{2}=-I+u_{i}\otimes U_{i}+v_{i}\otimes V_{i}
	\end{equation}%
	\begin{equation*} \label{H=GJ}
	G_{i}J=-JG_{i},\quad GU_{i}=0,\quad
	\end{equation*}%
	\begin{equation*} \label{g(GX,Y)=-g(X,GY)}
	g(K,G_{i}L)=-g(G_{i}K,L).
	\end{equation*}%
	2.On $\mathcal{U}_{i}\cap \mathcal{U}_{j}\neq \emptyset $ we have 
	\begin{eqnarray*}
		u_{j} &=&cu_{i}-dv_{i},\quad v_{j}=du_{i}+cv_{i},\; \\
		G_{j} &=&cG_{i}-dH_{i},\quad H_{j}=dG_{i}+cH_{i}
	\end{eqnarray*}%
	where $c$ and $d$ are functions on $\mathcal{U}_{i}\cap \mathcal{U}_{j}$
	with $c^{2}+d^{2}=1$ \cite{IK80}.
\end{defn}

In addition, we have 
\begin{align*}
du(K,L) &=g(K,GL)+(\sigma \wedge v)(K,L),~~~ \\
dv(K,L) &=g(K,HL)-(\sigma \wedge u)(K,L)
\end{align*}
where $\sigma (K)=g(\nabla _{K}U,V)$, and $\nabla $ being the Levi-Civita connection of $g$ \cite{IK80}.

Ishihara and Konishi \cite{IK79} defined following tensors ;
\begin{eqnarray*}
	S(K,L) &=&[G,G](K,L)+2g(K,GL)U-2g(K,HL)V \\
	&&+2(v(L)HK-v(K)HL)+\sigma (GL)HK \\
	&&-\sigma (GK)HL+\sigma (K)GHL-\sigma (L)GHK,\\
	T(K,L) &=&[H,H](K,L)-2g(K,GL)U+2g(K,HL)V \\
	&&+2(u(L)GK-u(K)GL)+\sigma (HK)GL \\
	&&-\sigma (HL)GK+\sigma (K)GHL-\sigma (L)GHK
\end{eqnarray*}
where 
\begin{equation*}
\lbrack G,G](K,L)=(\nabla _{GK}G)L-(\nabla _{GL}G)K-G(\nabla
_{K}G)L+G(\nabla _{L}G)K
\end{equation*}%
is the Nijenhuis torsion of $G$. Then they called an associated metric g
normal if $S=T=0.$ It is called IK-normality. In generally we consider whether or not the complex analogue of the real normal contact examples are IK-normal. The canonical example, complex Heisenberg group is not IK-normal \cite{BL2010}. Because it is not K\"ahler. In 2000 Korkmaz \cite{KB2000} gave a weaker definition.  
\begin{defn}
	$ M $ is normal if following two conditions are satisfied \cite{KB2000} : 
	\begin{enumerate}
		\item 	$S(K,L)=T(K,L)=0$ \ for all $K,L$ in $\mathcal{H},$ 
		\item  	$ S\left( K,U\right) =T\left( K,V\right) =0$ for all $K.$
		
	\end{enumerate}
\end{defn} 
A normal complex contact metric manifold is semi-K\"ahler and the complex Heisenberg group is normal. In this paper,  we use this notion of normality.

Korkmaz \cite{KB2000} obtained following equalities: 
\begin{align}
\nabla _{K}U &=-GK+\sigma (K)V,~~  \label{nabla KU} \\
\ \ \ \nabla _{K}V &=-HK-\sigma (K)U,~~  \label{nabla KV}\\
\nabla _{U}U &=\sigma (U)V,~~~\nabla _{U}V=-\sigma (U)U  \label{nablaUU} \\
\nabla _{V}U &=\sigma (V)V,~~~~\nabla _{V}V=-\sigma (V)U,~~  \notag \\
d\sigma (GK,GL) &=d\sigma (HK,HL)  \label{dsigma(GK,GL)} \\
&=d\sigma (L,K)-2u\wedge v(L,K)d\sigma (U,V).~  \notag
\end{align}
 
\begin{thm}\cite{KB2000}
	M is normal if and only if
	\begin{eqnarray}
	g((\nabla _{K}G)L,Z) &=&\sigma (K)g(HL,Z)+v(K)d\sigma (GZ,GL)
	\label{normallik (g(NablaXGY))} \\
	&&-2v(K)g(HGL,Z)-u(L)g(K,Z)  \notag \\
	&&-v(L)g(JK,Z)+u(Z)g(K,L)  \notag \\
	&&+v(Z)g(JK,L),  \notag\\
	g((\nabla _{K}H)L,Z) &=&-\sigma (K)g(GL,Z)-u(K)d\sigma (HZ,HL)
	\label{normallik (g(NablaXHY))} \\
	&&-2u(K)g(HGL,Z)+u(L)g(JK,Z)  \notag \\
	&&-v(L)g(K,Z)-u(Z)g(JK,L)  \notag \\
	&&+v(Z)g(K,L).  \notag
	\end{eqnarray}
\end{thm}

Also from above proposition we have 
\begin{align}
g((\nabla _{K}J)L,Z) &=u(K)(d\sigma (Z,GL)-2g(HL,Z))
\label{normallik (g(NablaXJY))} \\
&+v(K)(d\sigma (Z,HL)+2g(GL,Z)).  \notag
\end{align}%
Ishihara and Konishi \cite{IK79} proved a normality condition by the term of he covariant derivatives of $ G $ and $ H $. 	In \cite{vanliricci} we obtain following theorem for a normal complex contact metric manifold . 
\begin{thm}
	M is normal if and only
	if the covariant derivative of \ $G$ and $H$ have the following
	forms: 
	\begin{eqnarray}
	(\nabla _{K}G)L &=&\sigma (K)HL-2v(K)JL-u\left( L\right) K
	\label{Yeni normallik G} \\
	&&-v(L)JK+v(K)\left( 2JL_{0}-\left( \nabla _{U}J\right) GL_{0}\right)  \notag
	\\
	&&+g(K,L)U+g(JK,L)V  \notag \\
	&&-d\sigma (U,V)v(K)\left( u(L)V-v(L)U\right) , \notag\\
	(\nabla _{K}H)L &=&-\sigma (K)GL+2u(K)JL+u(L)JK  \label{Yeni normallik H} \\
	&&-v(L)X+u(K)\left( -2JL_{0}-\left( \nabla _{U}J\right) GL_{0}\right)  \notag
	\\
	&&-g(JK,L)U+g(K,L)V  \notag \\
	&&+d\sigma (U,V)u(K)\left( u(L)V-v(L)U\right)  \notag
	\end{eqnarray}
	where $K=K_{0}+u(K)U+v(K)V$ and $L=L_{0}+u(L)U+v(L)V, K_{0},L_{0}\in $ $%
	\mathcal{H}.$
\end{thm}

From this theorem we have

\begin{align*}
(\nabla _{K}J)L &=-2u\left(K\right) HL+2v(K)GL+u(K)\left( 2HL_{0}+\left(
\nabla _{U}J\right) L_{0}\right) \\
&+v(K)\left( -2GL_{0}+\left( \nabla _{U}J\right) JL_{0}\right) .
\end{align*}

\section{Fundamental Facts on Submanifolds of Normal Complex Contact Manifolds}
Let $ (\bar{M}^{4m+2}, \bar{G}, \bar{H}, \bar{J}, \bar{U}, \bar{V}, \bar{u}, \bar{v}, \bar{g} ) $ be a normal complex contact metric manifold, $ M $ be a $ (n+2)-$real dimensional complex submanifold of $ \bar{M} $  and $ \bar{U}, \bar{V} $ be tangent to $M$, where $ n $ must be even. The Gauss formula is given by 
\begin{eqnarray}
\bar{\nabla}_{K}L=\nabla_{K}L+\mathbf{h}(K,L) \label{gaussdenk}.
\end{eqnarray}
$ \mathbf{h}  $ is called the second fundamental form, and it is defined by:
\begin{equation*}\label{ikincitemelformacikhal}
\mathbf{h}(K,L) =\sum_{\alpha=1}^{r}(h^{\alpha}(K,L)N_{\alpha}+k^{\alpha}(K,L)\bar{J}N_{\alpha}).
\end{equation*}
where $ r=\frac{4m-n}{2} $. We have the  Wiengarten formulas which are given by 
\begin{eqnarray}
\bar{\nabla}_{K}N=-A_{N}K+\nabla^{\bot}_{K}N\label{weingarthen1}\\
\bar{\nabla}_{K}\bar{J}N=-A_{\bar{J}N}K+\nabla^{\bot}_{K}\bar{J}N \label{weingarthen2}
\end{eqnarray}
where $ A_{N} $ and $ A_{\bar{J}N} $ are fundamental forms related to $ N $ and $ \bar{J}N $. Also for $ s^{\alpha}(K) ,  t^{\alpha}(K)$ and $ \tilde{s}^{\alpha}(K) ,  \tilde{t}^{\alpha}(K)$ coefficients  
\begin{eqnarray*}
	\nabla^{\bot}_{K}N=\sum_{\alpha=1}^{r} \left(s^{\alpha}(K)N_{\alpha}+ t^{\alpha}(K)\bar{J}N_{\alpha} \right) \  \text{and}\  
	\nabla^{\bot}_{K}\bar{J}N=\sum_{\alpha=1}^{r}\left(\tilde{s}^{\alpha}(K)N_{\alpha}+ \tilde{t}^{\alpha}(K)\bar{J}N_{\alpha})\right.
\end{eqnarray*}
$ \bar{\nabla} , \nabla$ and $\nabla^{\bot}$ are the Riemannian, induced connection and induced normal connections on $ \bar{M} $, $ M $ and the normal bundle $ TM^{\bot}$ of $ M $, respectively. By easy computation we get following result. 
\begin{cor}
	For any  $ K,L \in \Gamma(TM) $ and $ N\in \Gamma(TM^{\bot}) $ we have  $ 
	\bar{g}(\mathbf{h}(K,L),N)=\bar{g}(A_{N}K,L). $
\end{cor} 
The mean curvature $ \mu $ of $ M $ is defined by $ \mu= \frac{trace\ \mathbf{h}}{dim\ M} $. 
$ M $ is a totally umbilical submanifold if  
\begin{equation}\label{totally umbl}
\mathbf{h}(K,L)=g(K,L)\mu 
\end{equation} for all $ K, L \in \Gamma(TM)$ .  \\
Since $ \bar{U},\bar{V} \in  \Gamma(TM) $ we can write $ TM=sp\{\bar{U}, \bar{V}\}\oplus sp\{\bar{U},\bar{V}\}^{\bot} $
where $ sp\{\bar{U}, \bar{V}\} $ is the distribution spanned
by $ \bar{U}, \bar{V}\ $ and $ sp\{\bar{U},\bar{V}\}^{\bot} $ is the complementary orthogonal distribution of $ sp\{\bar{U}, \bar{V}\} $ in M.  Then for any vector field  $ K $ is tangent to $ M $ we have  
 $ \bar{G}K\in sp\{\bar{U},\bar{V}\}^{\bot} $ and $\bar{H}K\in sp\{\bar{U},\bar{V}\}^{\bot} $.\\
Let define following projections; 
\begin{equation*}
P:\Gamma(TM)\rightarrow \Gamma(TM) \ \ ,\ \ Q:\Gamma(TM)\rightarrow \Gamma(TM)^{\bot} .
\end{equation*}
Then we can write 
\begin{equation} \label{GXifadesi}
\bar{G}K=PK+QK
\end{equation}
where $PK$ and $QK$ are the tangential and normal part of $ \bar{G}K $, respectively. Since $\bar{H}=\bar{G}\bar{J} $ we have 
\begin{equation} \label{HXifadesi}
\bar{H}K=P\bar{J}K+Q\bar{J}K.
\end{equation} 
Defined in this way $ P $ is an isomorphism on $\Gamma(TM)$ and $ Q $ is a normal valued 1-form on $\Gamma(TM)$. Therefore one can define two distributions for  $p \in M$ as follows 
\begin{eqnarray*}
	\mathcal{D}_{p}=ker\{Q\lvert_{sp\{\bar{U},\bar{V}\}^{\bot}}\}=\{K_{p} \in sp\{\bar{U},\bar{V}\}^{\bot}:Q(K_{p})=0\}\\
	\mathcal{D}^{\bot}_{p}=ker\{P\lvert_{sp\{\bar{U},\bar{V}\}^{\bot}}\}=\{K_{p} \in sp\{\bar{U},\bar{V}\}^{\bot}:P(K_{p})=0\}.
\end{eqnarray*}
The following result is directly obtained from the definition of $ \mathcal{D}_{p} $ and $ \mathcal{D}^{\bot}_{p} $.
\begin{prop}
	$ \mathcal{D}_{p} $ ve $ \mathcal{D}^{\bot}_{p} $ are orthogonal subspaces of $ T_{p}M $.
\end{prop}
On the other hand for any vector field $ N $ normal to $M$ we put 
\begin{equation} \label{GNifadesi}
\bar{G}N=BN+CN
\end{equation}
and 
\begin{equation} \label{HNifadesi}
\bar{H}N=B\bar{J}N+C\bar{J}N.
\end{equation} 
where $ BN,\ B\bar{J}N$ are tangential parts and $CN, \ C\bar{J}N$  are normal parts of $ \bar{G}N, \ \bar{H}N$, respectively. Therefore we have projections 
\begin{eqnarray*}
	B:\Gamma(TM^{\bot}) \rightarrow \Gamma(TM)\ and \ 	C:\Gamma(TM^{\bot}) \rightarrow \Gamma(TM^{\bot}).
\end{eqnarray*}

\section{Semi-invariant Submanifolds of Normal Complex Contact Metric Manifolds }
CR-submanifolds are important classes of complex submanifold theory. Similar to the definition of CR-submanifold, a semi-invariant submanifold of a Sasakian manifold was defined by Bejancu and Papaghuic \cite{bejancualmost}. We give an analogue definition for complex contact case.
\begin{defn} \label{semitanm1}
	Let $ (\bar{M}^{4m+2}, \bar{G}, \bar{H}, \bar{J}, \bar{U}, \bar{V}, \bar{u}, \bar{v}, \bar{g} ) $ be a normal complex contact metric manifold, and $ M $ be a complex submanifold of $ \bar{M} $. If the dimensions of $ \mathcal{D}_{p} $ and $ \mathcal{D}^{\bot}_{p} $ are constant along to $ M $ and  
	\begin{eqnarray*}
		\mathcal{D}:p\rightarrow \mathcal{D}_{p} \ \ ,\ \ 
		\mathcal{D}^{\bot}:p\rightarrow \mathcal{D}_{p}^{\bot}
	\end{eqnarray*}
	are differentiable then $ M $ is called a semi-invariant submanifold of $ \bar{M}$.
\end{defn}
Bejancu and Papaghuic proved two results (Proposition 1.1 and Proposition 1.2 in \cite{bejancualmost}) about invariance of these distributions. Similarly we obtain following propositions for complex contact case.
\begin{prop} \label{onrme1}
	The distribution $ \mathcal{D} $ is the maximal invariant distribution in $ sp\{\bar{U}, \bar{V}\}^{\bot} $; that is, we have  
	\begin{enumerate}
		\item $ \bar{G}\mathcal{D}_{p}=\bar{H}\mathcal{D}_{p}=\mathcal{D}_{p}$  
		\item If $ \mathcal{D}'_{p}\subset  $ $ sp\{\bar{U}, \bar{V}\}^{\bot} $ and $ \bar{G}\mathcal{D}'_{p}= \bar{H}\mathcal{D}'_{p}=\mathcal{D}'_{p} $ then we have $ \mathcal{D}'_{p}\subset \mathcal{D}_{p} $ .
	\end{enumerate}
\end{prop} 
\begin{prop} \label{onrm2}
	The distribution	$ \mathcal{D}^{\bot}_{p}   $  is the maximal anti-invariant distribution in $ sp\{\bar{U}, \bar{V}\}^{\bot} $; that is, we have  
	\begin{enumerate}
		\item $ \bar{G}\mathcal{D}^{\bot}_{p} \subset T^{\bot}_{p}M$  , $ \bar{H}\mathcal{D}^{\bot}_{p} \subset T^{\bot}_{p}M$
		\item If $ \mathcal{D}''_{p}\subset  $ $ sp\{\bar{U}, \bar{V}\}^{\bot} $ and $ \bar{G}\mathcal{D}''_{p}\subset T^{\bot}_{p}M $ , $ \bar{H}\mathcal{D}''_{p}\subset T^{\bot}_{p}M $ then we have  $ \mathcal{D}''_{p} \subset \mathcal{D}^{\bot}_{p} $ for any $ p\in M $.
	\end{enumerate}
\end{prop} 
In real Sasakian geometry, Bejancu and Papaghuic \cite{bejancusemi} gave an equivalent definition by using invariance of $ \mathcal{D}_{p},\ \mathcal{D}^{\bot}_{p} $. Similarly by considering the Preposition \ref{onrme1} and Preposition \ref{onrm2} we get an equivalent definition to Definition \ref{semitanm1}.
\begin{defn}
	Let $(\bar {M}^{4m+2},\bar{G}, \bar{H}, \bar{U}, \bar{V}, \bar{u}, \bar{v},g) $ be a normal complex contact metric manifold, $ M^{n+2} $ be a complex  submanifold of $\bar{M} $ and $ \bar{U} ,\bar{V} $ be tangent to $ M $. If following conditions are satisfied then $ M $ is called a semi-invariant submanifold.
	\begin{enumerate}
		\item $ TM=\mathcal{D}\oplus \mathcal{D}^{\bot}\oplus sp\{\bar{U},\bar{V}\} $.
		\item The distribution $ \mathcal{D} $ is invariant by $ \bar{G} $ and $ \bar{H} $; that is, $ \bar{G}\mathcal{D}=\mathcal{D} $ and  $ \bar{H}\mathcal{D}=\mathcal{D}$ .
		\item The distribution $ \mathcal{D}^{\bot} $ is anti-invariant by $ \bar{G} $ and $ \bar{H} $; that is, $\bar{G}\mathcal{D}^{\bot} \subset TM^{\bot} $  and $\bar{H}\mathcal{D}^{\bot}\subset {TM^{\bot}} $.
		
	\end{enumerate}
	
\end{defn}
Since  $ \bar{G}\bar{H}K=\bar{J}K $, for any vector field $ K $ in $ \Gamma(\mathcal{D}) $ or $ \Gamma(\mathcal{D}^{\bot})$ above conditions are also satisfied for $ \bar{J} $.\par 
Let $ M $  be a semi-invariant submanifold of a normal complex contact metric manifold $ \bar{M} $. If $ dim\mathcal{D}=0 $ then $ M $ is called an anti-invariant submanifold of $ \bar{M} $, and if $ dim\mathcal{D}^{\bot}=0 $ then $ M $ is called an invariant submanifold of $ \bar{M} $. If $ \bar{G}\mathcal{D}^{\bot}= \bar{H}\mathcal{D}^{\bot}=TM^{\bot}$, then $ M $ is called generic submanifold of $ \bar{M} $ \cite{yanoCRkitap}.\par
For a semi-invariant submanifold $ M $ of a normal complex contact metric manifold  $ \bar{M} $, the projection morphisms of $ TM $ to $ \mathcal{D} $ and $ \mathcal{D}^{\bot} $ are denoted by $ \phi $ and $ \psi $, respectively . Then for all $ K\in \Gamma(TM) $ we can write 
\begin{equation} \label{altmanifoldX}
K=\phi K+ \psi K+ \bar{u}(K)\bar{U}+\bar{v}(K)\bar{V}
\end{equation}
where $ \phi K$ and $ \psi K $ are tangential and normal parts of $ K $, respectively. 
Also we have
\begin{equation}\label{altmanifoldJX}
\bar{J}K=\phi \bar{J} K+ \psi \bar{J}K+ \bar{v}(K)\bar{U}-\bar{u}(K)\bar{V}.
\end{equation}
Similarly for $ N, \bar{J}N\in TM^{\bot} $ we have 
\begin{equation*}
N=t N+ f N \ \ \text{and} \ \ \bar{J}N=t\bar{J}N+f\bar{J}N
\end{equation*}
where $ t N ,t\bar{J}N$ is tangential part , and $ fN, f\bar{J}N $ is the normal part of $ N, \bar{J}N $. On the other hand from (\ref{GNifadesi}) and (\ref{HNifadesi}) we have $ BN \in \Gamma(\mathcal{D}^\bot), B\bar{J}N \in \Gamma(\mathcal{D}^\bot) $,  $ CN \in \Gamma(TM^{\bot}) and C\bar{J}N \in \Gamma(TM^{\bot}) $. Thus we obtain an $ f-$structure on the normal bundle by following same steps with the proof of Proposition 1.3 in \cite{bejancusemi}.\\\\  
From now on we will denote a semi-invariant submanifold of  a normal complex contact metric manifold by M.
 
\begin{prop}
	On the normal bundle of $M$ there exist an $ f- $structure $ C $.
\end{prop}

For $ M $ we have following decomposition of normal space $ TM^{\bot} $: 
\begin{equation*}
TM^{\bot}=\bar{G}\mathcal{D}^{\bot}\oplus \bar{H}\mathcal{D}^{\bot}\oplus \bar{J}\mathcal{D}^{\bot}\oplus \vartheta.
\end{equation*}
We can take an orthonormal frame 
\begin{equation*}
\{e_{1},e_{2},...,e_{m},\bar{G}e_{1},\bar{G}e_{2},...,\bar{G}e_{m}, \bar{H}e_{1},\bar{H}e_{2},...,\bar{H}e_{m}, \bar{J}e_{1},\bar{J}e_{2},...,\bar{J}e_{m}, \bar{U}, \bar{V} \} 
\end{equation*}
of $ \bar{M} $ such that $ \{e_{1},e_{2},...,e_{n}\} $ are tangent to $ M $. Therefore the set $ \{e_{1},e_{2},...,$ $e_{n}, e_{n+1}=\bar{U}, e_{n+2}=\bar{V}\} $ is an orthonormal frame of $ M $. We can consider $ \{e_{1},e_{2},...,e_{n}\} $  such that $ \{e_{1},e_{2},...,e_{p}\} $ is an orthonormal frame of $ \mathcal{D} ^{\bot}$, $ \{e_{p+1},$ $e_{p+2},...,e_{n}\}$ is an orthonormal frame of $ \mathcal{D} $ . Moreover we can take $  \{e_{n+3},...,$ $e_{4m-n}\} $ as an orthonormal frame of $ TM^{\bot} $ such that $  \{e_{n+3},...,e_{n+2+3p}\} $ is an orthonormal frame of $\bar{G}\mathcal{D}^{\bot}\oplus \bar{H}\mathcal{D}^{\bot}\oplus \bar{J}\mathcal{D}^{\bot}  $ and $  \{e_{n+3+3p},e_{n+4+3p},...,e_{4m+2}\} $ is an orthonormal frame of $ \vartheta $. From the definition of semi-invariant manifold we can take $ e_{n+3}=\bar{G}e_{1},\  e_{n+4}=\bar{G}e_{2}\ ,...,\ e_{n+2+p}=\bar{G}e_{p},\ e_{n+3+p}=\bar{H}e_{1},\ e_{n+4+p}=\bar{H}e_{2}\ ,...,\ e_{n+2+2p}=\bar{H}e_{p},\ e_{n+3+2p}=\bar{J}e_{1},\ e_{n+4+2p}=\bar{J}e_{2}\ ,...,\ e_{n+2+3p}=\bar{J}e_{p}   $. Therefore we have following orthonormal basis: 
\begin{eqnarray*}
	\mathcal{D}&=&sp\{e_{\frac{p+1}{4}},e_{\frac{p+2}{4}},...,e_{\frac{n-3}{4}},\bar{G}e_{\frac{p+1}{4}},\bar{G}e_{\frac{p+5}{4}},...,\bar{G}e_{\frac{n-3}{4}}, \bar{H}e_{\frac{p+1}{4}},\\&&\bar{H}e_{\frac{p+2}{4}},...,\bar{H}e_{\frac{n-3}{4}}, \bar{J}e_{\frac{p+1}{4}},\bar{J}e_{\frac{p+2
		}{4}},...,\bar{J}e_{\frac{n-3}{4}} \} \\
	\mathcal{D}^{\bot}&=&sp\{e_{1},e_{2},...,e_{p}\}
\end{eqnarray*}
and 
\begin{eqnarray*}
	\bar{G}\mathcal{D}^{\bot}\oplus \bar{H}\mathcal{D}^{\bot}\oplus \bar{J}\mathcal{D}^{\bot}&=&sp \{\bar{G}e_{1},\bar{G}e_{2},...,\bar{G}e_{p}, \bar{H}e_{1},\bar{H}e_{2},...,\bar{H}e_{p},\\&& \bar{J}e_{1},\bar{J}e_{2},...,\bar{J}e_{p}\}
	\\
	\vartheta&=& sp\{e_{p+1},e_{p+2},...,e_{\frac{4m-n+3p}{4}},\bar{G}e_{p+1},\\&&\bar{G}e_{p+2},...,\bar{G}e_{\frac{4m-n+3p}{4}}, \bar{H}e_{p+1},\bar{H}e_{p+2},\\
	&&...,\bar{H}e_{\frac{4m-n+3p}{4}}, \bar{J}e_{p+1},\bar{J}e_{p+2},...,\bar{J}e_{\frac{4m-n+3p}{4}} \}.
\end{eqnarray*}
For $ M $ we compute covariant derivatives of $ \bar{G}, \bar{H}, \bar{J}  $ by given tangential and normal components. From (\ref{normallik (g(NablaXGY))}), (\ref{normallik (g(NablaXHY))},) (\ref{normallik (g(NablaXJY))}), (\ref{gaussdenk}), (\ref{weingarthen1}), (\ref{weingarthen2}), (\ref{GXifadesi}), (\ref{HXifadesi}), (\ref{GNifadesi}) and (\ref{HNifadesi})  and by easy computation we have following lemmas. 
\begin{lem}
	For any  $ K,L \in \Gamma(TM) $  we have 
	\begin{align}
	\phi\nabla_{K}PL-\phi A_{QL}K&=P\nabla_{K}L-\bar{u}(L)\phi K+\bar{\sigma}(K)P\bar{J}L \\
	&-2\bar{v}(K)\phi\bar{J}L-\bar{v}(L)\phi \bar{J}K+2\bar{v}(K)\phi \bar{J}L_{0}\notag\\
	&-\bar{v}(K)(\phi\nabla_{\bar{U}}\bar{J}PL_{0} 	-\bar{J}\phi\nabla_{\bar{U}}PL_{0}\notag\\
	&-\phi A_{\bar{J}QL_{0}}\bar{U}+\bar{J}\phi A_{QL_{0}}\bar{U}), \notag \\
	\psi\nabla_{K}PL-\psi _{QL}K&= B\mathbf{h}(K,L)+\bar{\sigma}(K)Q\bar{J}L\\
	&-2\bar{v}(K)\psi\bar{J}L-\bar{u}(L)\psi K-\bar{v}(L)\psi\bar{J}K \notag\\&+\bar{v}(K)\psi\bar{J}L_{0}-\bar{v}(K)(\psi\nabla_{\bar{U}}\bar{J}PL_{0}-\bar{J}\psi\nabla_{\bar{U}}PL_{0}\notag\\
	&-\psi A_{\bar{J}QL_{0}}\bar{U}+\bar{J}\psi A_{QL_{0}}\bar{U}-B\bar{J}C\mathbf{h}(\bar{U},PL_{0})),\notag \\
	\bar{u}(\nabla_{K}PL-A_{QL}K)&=\bar{g}(\phi K, \phi L)+\bar{g}(\psi K, \psi L) \label{lemma1.1.3}\\
	&+(d\bar{\sigma}(\bar{U},\bar{V})-2)\bar{v}(K)\bar{v}(L)-\bar{v}(K)(\bar{u}(\nabla_{\bar{U}}\bar{J}PL_{0}\notag \\ &-A_{\bar{J}QL_{0}}\bar{U})+\bar{v}(A_{QL_{0}}\bar{U}-\nabla_{\bar{U}}PL_{0}))\notag,
	\end{align}
	\begin{align}
	\bar{v}(\nabla_{K}PL-A_{QL}K)&=\bar{g}(\phi\bar{J} K, \phi L)+\bar{g}(\psi \bar{J} K, \psi L)  \label{lemma1.1.4}\\
	&-(d\bar{\sigma}(\bar{U},\bar{V})-2)\bar{v}(K)\bar{u}(L)-\bar{v}(K)(\bar{v}(\nabla_{\bar{U}}\bar{J}PL_{0}\notag\\
	&-A_{\bar{J}QL_{0}}\bar{U})+\bar{u}(\nabla_{\bar{U}}PL_{0}+A_{QL_{0}}\bar{U})\notag,\\
	\mathbf{h}(K,PL)-C\mathbf{h}(K,L)+Q\nabla_{K}L&=\nabla^{\bot}_{K}QL-\bar{v}(K)(\mathbf{h}(\bar{U}, \bar{J}PL_{0}) \label{lemma1.1.5}\\&-Q\bar{J}B\mathbf{h}(\bar{U},PL_{0})-C\bar{J}C\mathbf{h}(\bar{U},PL_{0}) \notag \\ &+\nabla_{\bar{U}}^{\bot}\bar{J}QL_{0}-\bar{J}\nabla_{\bar{U}}^{\bot QL_{0}}).\notag
	\end{align}
\end{lem}
\begin{lem}
	 For arbitrary vector fields $ K $ and $ L  $ on $ M $ we have 
	\begin{align}
	\phi\nabla_{K}P\bar{J}L-\phi A_{Q\bar{J}L}K&=P\bar{J}\nabla_{K}L-\bar{\sigma}(K)PL+2\bar{u}(K)\phi\bar{J}L\\
	&+\bar{u}(L)\phi \bar{J}K-\bar{v}(L)\phi K - 2\bar{u}(K)\phi\bar{J}L_{0} \notag\\
	&-\bar{u}(K)(\phi\nabla_{\bar{U}}\bar{J}PL_{0}	-\bar{J}\phi\nabla_{\bar{U}}PL_{0}\notag \\&-\phi A_{\bar{J}QL_{0}}\bar{U}+\bar{J}\phi A_{QL_{0}}\bar{U}), \notag\\
	\psi\nabla_{K}P\bar{J}L-\psi A_{Q\bar{J}L}K&= B\bar{J}\mathbf{h}(K,L)-\sigma(K)QL+2\bar{u}(K)\psi\bar{J}L\label{lemma1.1.2k}\\
	&+\bar{u}(L)\psi \bar{J}K-\bar{v}(L)\psi K-2\bar{u}(K)\psi\bar{J}L_{0} \notag\\
	&-\bar{u}(K)(\psi\nabla_{\bar{U}}\bar{J}PL_{0}-\bar{J}\psi\nabla_{\bar{U}}PL_{0}\notag \\&-\psi A_{\bar{J}QL_{0}}\bar{U}+\bar{J}\psi A_{QL_{0}}\bar{U}-B\bar{J}C\mathbf{h}(\bar{U},PL_{0})),\notag \\
	\bar{u}(\nabla_{K}P\bar{J}L-A_{Q\bar{J}L}K)&=-\bar{g}(\phi \bar{J}K, \phi L)-\bar{g}(\psi \bar{J}K, \psi L) \\
	&-(d\sigma(\bar{U},\bar{V})-2)\bar{v}(K)\bar{v}(L)+\bar{u}(K)(-\bar{u}(\nabla_{\bar{U}}\bar{J}PL_{0}\notag \\ &-A_{\bar{J}QL_{0}}\bar{U})+\bar{v}(\nabla_{\bar{U}}PL_{0}+A_{QL_{0}}\bar{U})\notag,\\
	\bar{v}(\nabla_{K}P\bar{J}L-A_{Q\bar{J}L}K)&=\bar{g}(\phi K, \phi L)+\bar{g}(\psi K, \psi L) \\
	&+(d\sigma(\bar{U},\bar{V})-2)\bar{u}(K)\bar{u}(L)-\bar{u}(K)(\bar{u}(\nabla_{\bar{U}}PL_{0}\notag \\ &+A_{QL_{0}}\bar{U})+\bar{v}(\nabla_{\bar{U}}\bar{J}PL_{0}-A_{\bar{J}QL_{0}}\bar{U})\notag \\ &+\bar{v}(\nabla_{\bar{U}}PL_{0})+\bar{u}(A_{\bar{J}QL_{0}}\bar{U})+\bar{v}(A_{\bar{J}QL_{0}}\bar{U})),\notag\\
	\mathbf{h}(K,P\bar{J}L)-C\bar{J}\mathbf{h}(K,L)&=-Q\bar{J}\nabla_{K}L-\nabla^{\bot}_{K}Q\bar{J}L\label{lemma1.1.5k}\\&-\bar{u}(K)(\mathbf{h}(\bar{U}, \bar{J}PL_{0})-Q\bar{J}B\mathbf{h}(\bar{U},PL_{0})\notag \\ &-C\bar{J}C\mathbf{h}(\bar{U},PL_{0})+\nabla_{\bar{U}}^{\bot}\bar{J}QL_{0}-\bar{J}\nabla_{\bar{U}}^{\bot QL_{0}}) \notag.
	\end{align} 	
\end{lem}
\begin{lem}
	 For any  $ K,L \in \Gamma(TM) $  we have 
	\begin{align*}
	\phi\nabla_{K}BN-\phi A_{CN}K-PA_{N}K&= \bar{v}(K)(\phi A_{\bar{J}BN}\bar{U}+\phi\bar{J}\nabla_{\bar{U}}BN\\
	&-\phi A_{\bar{J}CN}\bar{U}-\phi\bar{J}A_{CN}\bar{U}),\notag\\
	\psi\nabla_{K}BN-\psi A_{CN}K -B\nabla^{\bot}_{K}N&=\bar{\sigma}(K)B\bar{J}N+\bar{v}(K)(\psi A_{\bar{J}BN}\bar{U}
	\\&+ \psi\bar{J}\nabla_{\bar{U}}BN+B\bar{J}C\mathbf{h}(\bar{U},BN)+\psi A_{\bar{J}CN}\bar{U}\notag\\&-\psi\bar{J}A_{CN}\bar{U}+B\bar{J}C\nabla_{U}^\bot CN),\notag
	\end{align*}
	\begin{align*}
	\bar{u}(\nabla_{K}BN)-\bar{u}(A_{CN}K)&=\bar{v}(K)[\bar{u}(A_{\bar{J}BN}\bar{U})+\bar{v}(\nabla_{U}BN)\\&+\bar{u}(A_{C\bar{J}N}\bar{U})+\bar{v}(A_{CN}\bar{U})],\notag\\
	\bar{v}(\nabla_{K}BN)-\bar{v}(A_{CN}K)&=\bar{v}(K)[\bar{v}(A_{\bar{J}BN}\bar{U})-\bar{u}(\nabla_{U}BN)\\&+\bar{v}(A_{\bar{J}CN}\bar{U})+\bar{u}(A_{CN}\bar{U})],\notag
	\end{align*}
	\begin{align*}
	\mathbf{h}(K,BN)+\nabla_{K}^\bot CN-QA_{N}K&=C\nabla_{K}^\bot N+\bar{\sigma}(K)C\bar{J}N-\bar{v}(K)[\nabla_{\bar{U}}^\bot B\bar{J}N\\&-Q\bar{J}B\mathbf{h}(\bar{U},BN)+\nabla_{\bar{U}}^\bot\bar{J}CN \notag \\ &-Q\bar{J}C\nabla_{\bar{U}}^\bot CN-C\bar{J}C\nabla_{\bar{U}}^\bot CN]. \notag
	\end{align*}
\end{lem} 
\begin{lem}
	For any  $ K,L \in \Gamma(TM) $  we have 
	\begin{align*}
	\phi\nabla_{K}B\bar{J}N-\phi A_{C\bar{J}N}K+PA_{N}K&= \bar{u}(K)(-\phi A_{\bar{J}BN}\bar{U}\\
	&+\phi\bar{J}\nabla_{U}BN-\phi A_{\bar{J}CN}\bar{U}+\phi\bar{J}A_{CN}\bar{U}),\notag\\
	\psi\nabla_{K}B\bar{J}N-\psi A_{C\bar{J}N}K -B\bar{J}\nabla^{\bot}_{K}N&=\bar{\sigma}(K)BN+\bar{u}(K)(-\psi A_{\bar{J}BN}\bar{U}\\&- \psi\bar{J}\nabla_{\bar{U}} BN-B\bar{J}C\mathbf{h}(\bar{U},BN)-\psi A_{\bar{J}CN}\bar{U} \notag \\ &+\psi\bar{J}A_{CN}\bar{U}-B\bar{J}C\nabla_{U}^\bot CN),\notag\\ 
	\bar{u}(\nabla_{K}B\bar{J}N)-\bar{u}(A_{C\bar{J}N}K)&=-\bar{u}(K)[\bar{u}(A_{\bar{J}BN}\bar{U})\\&+\bar{v}(\nabla_{U}BN)+\bar{u}(A_{C\bar{J}N}\bar{U})+\bar{v}(A_{CN}\bar{U})],\notag\\
	\bar{v}(\nabla_{K}B\bar{J}N)-\bar{v}(A_{C\bar{J}N}K)&=-\bar{u}(K)[\bar{v}(A_{\bar{J}BN}\bar{U})\\&-\bar{u}(\nabla_{U}BN)+\bar{v}(A_{\bar{J}CN}\bar{U})-\bar{u}(A_{CN}\bar{U})],\notag\\
	\mathbf{h}(K,B\bar{J}N)+\nabla_{K}^\bot C\bar{J}N-QA_{N}K&=C\bar{J}\nabla_{K}^\bot N+\bar{\sigma}(K)CN+\bar{u}(K)[\nabla_{\bar{U}}^\bot B\bar{J}N\\&-Q\bar{J}B\mathbf{h}(\bar{U},BN)-C\bar{J}C\mathbf{h}(\bar{U},BN)\notag\\&+\nabla_{\bar{U}}^\bot\bar{J}CN -Q\bar{J}C\nabla_{\bar{U}}^\bot CN-C\bar{J}C\nabla_{\bar{U}}^\bot CN]. \notag	
	\end{align*}
\end{lem} 

As we know the covariant derivatives of structure vector fields are important. In the following lemma we give 
the covariant derivatives of $ \bar{U} $ and $ \bar{V} $ with $ \nabla $ on $M$. 
\begin{lem}
	For any  $ K,L \in \Gamma(TM) $ we have	
	\begin{eqnarray*}
		\nabla_{K}\bar{U}&=&-PK+\bar{\sigma}(K)\bar{V},\ \ \ \mathbf{h}(K,\bar{U})=-QK \label{almanifoldnablaKU}\\
		\nabla_{K}\bar{V}&=&-P\bar{J}K-\bar{\sigma}(K)\bar{U},\ \ \ \mathbf{h}(K,\bar{V})=-Q\bar{J}K.\label{almanifoldnablaKV}
	\end{eqnarray*}
\end{lem}
\begin{proof}
	From (\ref{nabla KU}) and (\ref{gaussdenk}) we get 
	\begin{eqnarray*}
		-\bar{G}K+\bar{\sigma}(K)\bar{V}=\nabla_{K}\bar{U}+\mathbf{h}(K,\bar{U}),
	\end{eqnarray*}
	and by consider tangent and normal components we obtain (\ref{almanifoldnablaKU}). Similarly from (\ref{nabla KV}) and (\ref{gaussdenk}) we get (\ref{almanifoldnablaKV}).
\end{proof}
Also from these lemmas, we get following corollaries. 
\begin{cor}
	For $ M $ we have 
	\begin{eqnarray*}
		\mathbf{h}(K,\bar{U})&=&\mathbf{h}(K,\bar{V})=0\\ 
		\nabla_{K}\bar{U}&=&-PK+\bar{\sigma}(K)\bar{V},\ \ \nabla_{K}\bar{V}=-P\bar{J}K+\bar{\sigma}(K)\bar{U} \label{altmanifoldnablaxU}
	\end{eqnarray*}
	for all $ K\in \Gamma(\mathcal{D}) $, and 
	\begin{eqnarray*}
		\mathbf{h}(K,\bar{U})&=&-QK,\ \mathbf{h}(K,\bar{V})=-Q\bar{J}K \label{h(X,U)}\\ 
		\nabla_{K}\bar{U}&=&\bar{\sigma}(K)\bar{V},\ \ \nabla_{K}\bar{V}=-\bar{\sigma}(K)\bar{U}
	\end{eqnarray*}
	for all $ K\in \Gamma(\mathcal{D}^\bot) $.
\end{cor}
\begin{cor}
	For $ M $ we have 
	\begin{eqnarray*}
		\mathbf{h}(\bar{U},\bar{U})&=&\mathbf{h}(\bar{V},\bar{V})=\mathbf{h}(\bar{U},\bar{V})=0\\
		\nabla_{\bar{U}}\bar{U}&=&\bar{\sigma}(\bar{U})\bar{V},\ \ \nabla_{\bar{V}}\bar{U}=\bar{\sigma}(\bar{V})\bar{V}\\  \nabla_{\bar{U}}\bar{V}&=&-\bar{\sigma}(\bar{U})\bar{U}, \ \ \nabla_{\bar{V}}\bar{V}=-\bar{\sigma}(\bar{V})\bar{U}.
	\end{eqnarray*}
\end{cor}

\section{Integrability of Distributions}
In the submanifold theory integrability of distributions is an important notion. In this work we have two distributions, $\mathcal{D}$ and $ \mathcal{D}^{\bot}  $. In this section we give some result about integrability of $ \mathcal{D},\ \mathcal{D}^{\bot},\ \mathcal{D}\oplus \mathcal{D}^{\bot}, \ \mathcal{D}\oplus sp\{\bar{U},\bar{V}\}$ and $\mathcal{D}^{\bot}\oplus sp\{\bar{U},\bar{V}\}$. Although the horizontal distribution $ \mathcal{H} $ is never involute, as we shall see some of above distributions are involute on $ M. $ 

\begin{lem}
	For $ M $ we have 
	\begin{eqnarray}
	\bar{g}(A_{\bar{G}K}L,Z)=\bar{g}(A_{\bar{G}L}K,Z) \label{g(AGXL,Z)}\\
	\bar{g}(A_{\bar{H}K}L,Z)=\bar{g}(A_{\bar{H}L}K,Z)\label{g(AHXY,Z)}\\
	\bar{g}(A_{\bar{J}K}L,Z)=\bar{g}(A_{\bar{J}L}K,Z) \label{g(AJXY,Z)}
	\end{eqnarray}
	for all $ K,L \in \Gamma(\mathcal{D}) $ , $ Z $ is tangent to $ M $ and $ Z\notin sp\{\bar{U},\bar{V}\}$.
\end{lem}
\begin{proof}
	Let  $ K,L \in \Gamma(\mathcal{D}) $ and $ Z\in \Gamma(TM) $. Since $ \bar{G}K=QK\in \Gamma(TM^{\bot}) $, we have 
	\begin{eqnarray} \label{integlemma1ispat1}
	\bar{\nabla}_{L}\bar{G}K=-A_{\bar{G}K}L+\nabla_{L}^{\bot}\bar{G}K
	\end{eqnarray}
	and 
	\begin{eqnarray*}\label{integlemma1ispat2}
		\bar{\nabla}_{L}Z=\nabla_{L}Z+\mathbf{h}(L,Z).
	\end{eqnarray*}
	Thus we get
	\begin{eqnarray*}\label{g(h(x,y),n)}
		\bar{g}(\bar{\nabla}_{L}Z,\bar{G}K)=\bar{g}(\mathbf{h}(L,Z),\bar{G}K)\\
	\end{eqnarray*}
	and since $ \bar{G}K\in \Gamma(TM^{\bot})$ then $ 	\bar{g}(\bar{\nabla}_{L}Z,\bar{G}K)+\bar{g}(Z,\bar{\nabla}_{L}\bar{G}K)=0$ and therefore from (\ref{integlemma1ispat1}) we get 
	\begin{eqnarray*}
		\bar{g}(A_{\bar{G}K}L,Z)=\bar{g}(\mathbf{h}(L,Z),\bar{G}K).
	\end{eqnarray*}
	In addition since $ \mathbf{h} $ is symmetric and from (\ref{integlemma1ispat1}) we have 
	\begin{eqnarray*}
		\bar{g}(A_{\bar{G}K}L,Z)&=&-\bar{g}(\bar{G}\bar{\nabla}_{Z}L,K)\\
		&=& \bar{g}((\bar{\nabla}_{Z}\bar{G})L,K)-\bar{g}(\bar{\nabla}_{Z}\bar{G}L,K).
	\end{eqnarray*}
	From (\ref{normallik (g(NablaXGY))}) and (\ref{dsigma(GK,GL)}) we have 
	\begin{eqnarray*} \label{integlemma1ispat3}
		\bar{g}((\bar{\nabla}_{Z}\bar{G})L,K)=\bar{g}(d\sigma(L,K)\bar{V},Z)
	\end{eqnarray*}
	and so we get 
	\begin{eqnarray*}
		\bar{g}(A_{\bar{G}K}L,Z)&=&\bar{g}(d\sigma(L,K)\bar{V},Z)-\bar{g}(\bar{\nabla}_{Z}\bar{G}L,K).
	\end{eqnarray*}
	On the other hand since $ \bar{g}(\bar{G}L,K) =0$ and from (\ref{gaussdenk}) we have 
	\begin{align*}
	\bar{g}(A_{\bar{G}K}L,Z)&=\bar{g}(d\sigma(L,K)\bar{V},Z)+\bar{g}(\bar{\nabla}_{Z}K, \bar{G}L)\\
	&=\bar{g}(d\sigma(L,K)\bar{V},Z)+\bar{g}(\nabla_{Z}K+\mathbf{h}(Z,K),\bar{G}L)	\\
	&= \bar{g}(d\sigma(L,K)\bar{V},Z)+\bar{g}(\mathbf{h}(Z,K),\bar{G}L)
	\end{align*}
	and thus, from (\ref{integlemma1ispat3}) we get 
	\begin{equation*}
	\bar{g}(A_{\bar{G}K}L,Z)=\bar{g}(d\sigma(L,K)\bar{V},Z)+	\bar{g}(A_{\bar{G}L}K,Z).
	\end{equation*}
	If $ Z\notin sp\{\bar{U},\bar{V}\}  $ we get (\ref{g(AGXL,Z)}). By following same steps one can show (\ref{g(AHXY,Z)}), (\ref{g(AJXY,Z)}). 
\end{proof}

\begin{lem}
	For all $ K,L \in \Gamma (\mathcal{D}^{\bot}) $ we have $ [K,L]\in \Gamma(\mathcal{D}\oplus \mathcal{D}^{\bot}) $.
\end{lem}
\begin{proof}
	Let $ K,L \in \Gamma (D^{\bot}) $. Then we have 
	\begin{eqnarray*}
		\bar{g}([K,L],\bar{U})&=&\bar{g}(\bar{\nabla}_{K}L-\bar{\nabla}_{L}K,\bar{U})\\
		&=&-\bar{g}(\bar{\nabla}_{K}\bar{U},L)+\bar{g}(\bar{\nabla}_{L}\bar{U},K).
	\end{eqnarray*}
	Therefore from (\ref{nabla KU}) we have $ \bar{g}([K,L],\bar{U})=0 $. Also  $\bar{g}([K,L],\bar{V})=0 $ can be showed by similar way. So we obtain $ [K,L]\in \Gamma(D\oplus D^{\bot}) $.
	
\end{proof}

\begin{thm} \label{Teorem 2.2 }
	The anti-invariant distribution is involutive. 
\end{thm}
\begin{proof}
	Let $ K,L \in \Gamma (\mathcal{D}^{\bot}) $. From (\ref{Yeni normallik G}) we have 
	\begin{eqnarray*}
		(\bar{\nabla}_{K}\bar{G})L=\bar{\sigma}(K)\bar{H}L+\bar{g}(K,L)\bar{U}.
	\end{eqnarray*}
	On the other hand $ \bar{G}L\in\Gamma(D) $ and from (\ref{gaussdenk}) and (\ref{weingarthen1}) we have 
	\begin{eqnarray} \label{distribusyonteorem1ispat1}
	-A_{\bar{G}L}K+\nabla_{K}^{\bot}L-\bar{G}\nabla_{K}L-\bar{G}\mathbf{h}(K,L)=\bar{\sigma}(K)\bar{H}L+\bar{g}(K,L)\bar{U}. 
	\end{eqnarray}
	Substituting $ L $ by $ K $ in (\ref{distribusyonteorem1ispat1}) and thus   subtracting the obtained relations we get 
	\begin{eqnarray*}
		-\bar{G}[K,L]=A_{\bar{G}L}K-A_{\bar{G}K}L-\nabla_{K}^{\bot}\bar{G}L-\nabla_{L}^{\bot}\bar{G}K+\bar{\sigma}(K)\bar{H}L-\bar{\sigma}(L)\bar{H}K.
	\end{eqnarray*}
	Now we take an arbitrary normal section $ N\in \Gamma(\vartheta) $ and,  by using (\ref{normallik (g(NablaXGY))}) and (\ref{weingarthen1}) we have 
	\begin{eqnarray} \label{distribusyonteorem1ispat2}
	\bar{g}(\nabla_{L}^{\bot}\bar{G}K,N)=-\bar{g}(A_{\bar{G}N}L,K).
	\end{eqnarray}
	Substituting $ L $ by $ K $ in (\ref{distribusyonteorem1ispat2}) and, subtracting the obtained relations , since $ A_{\bar{G}N} $ is symmetric  we have 
	\begin{eqnarray*}
		\bar{g}(\nabla_{K}^{\bot}\bar{G}L-\nabla_{L}^{\bot}\bar{G}K,N)=0.
	\end{eqnarray*}
	Hence $  \nabla_{K}^{\bot}\bar{G}L-\nabla_{L}^{\bot}\bar{G}K \in \bar{G}\mathcal{D}^{\bot}\oplus \bar{H}\mathcal{D}^{\bot}\oplus \bar{J}\mathcal{D}^{\bot} $.  On the other hand for $ Z\in \Gamma(\mathcal{D}) $ from (\ref{distribusyonteorem1ispat2}) we have
	\begin{eqnarray*}
		\bar{g}(-\bar{G}[K,L],\bar{G}Z)=0
	\end{eqnarray*}
	and therefore we get 
	\begin{eqnarray*}
		\bar{g}([K,L],\bar{G}^{2}Z)=\bar{g}([K,L],Z)=0.
	\end{eqnarray*}
	So we obtain $ [K,L] \in \Gamma(\mathcal{D}) $.
\end{proof}
\begin{thm}
	$ \mathcal{D}^\bot \oplus sp\{\bar{U},\bar{V}\} $ distribution is involute. 
\end{thm}
\begin{proof}
	Let $ K\in \Gamma(\mathcal{D}^\bot) $ and  $L\in \Gamma(\mathcal{D}) $. Then from (\ref{nabla KU}) we have 
	\begin{eqnarray*}
		\bar{g}([K,\bar{U}],L)=-\bar{g}(\bar{\nabla}_{\bar{U}}K,L).
	\end{eqnarray*}
	Now let take $ Z\in \Gamma(\mathcal{D}) $ such that $L= \bar{G}Z $ and by using (\ref{Yeni normallik G})  we have 
	\begin{eqnarray*}
		(\bar{\nabla}_{\bar{U}}H)Z=\sigma(\bar{U})\bar{H}Z
	\end{eqnarray*}
	and from (\ref{gaussdenk}) we get 
	\begin{eqnarray*}
		\bar{g}([K,U],L)=\bar{g}(\bar{\nabla}_{\bar{U}}\bar{G}Z,K)=-\bar{g}(\nabla_{\bar{U}}Z,\bar{G}K)=0.
	\end{eqnarray*}
	Therefore $ [K,\bar{U}] \in \mathcal{D}^{\bot}\oplus sp\{\bar{U},\bar{V}\}$. Following by same steps one can show the $ [K,\bar{U}] \in \mathcal{D}^{\bot}\oplus sp\{\bar{U},\bar{V}\}$.
	Consequently by consider (\ref{Teorem 2.2 }) the theorem is proved. 
\end{proof}

\begin{defn}
	 If $ M $ is neither an invariant submanifold (i.e $ dim \mathcal{D}^{\bot}=0 $ ) nor an anti-invariant submanifold (i.e $ dim \mathcal{D}=0 $  ), then it is called a proper semi-invariant submanifold.  
\end{defn}

\begin{thm}
	The invariant distribution is never involute.   
\end{thm}
\begin{proof}
	For $ K,L \in \Gamma(\mathcal{D}) $ from (\ref{nabla KU}) we get 
	\begin{equation*}
	\bar{g}([K,L],\bar{U})=2\bar{g}(\bar{G}K,L)
	\end{equation*}
	and from (\ref{nabla KV}) we have 
	\begin{equation*}
	\bar{g}([K,L],\bar{U})=2\bar{g}(\bar{G}K,L). 
	\end{equation*}
	Let choose $ L=\bar{H}K $ for all $ L\in \Gamma(\mathcal{D}) $ such that $ \bar{H}K $ is a unit vector field. Thus the second fundamental form can not vanish. So $ \mathcal{D} $ is not involute. 
\end{proof}
From this theorem we have :
\begin{cor}
	Let $ M $ be a proper semi-invariant submanifold. Then the distribution $ \mathcal{D}\oplus \mathcal{D}^{\bot} $ is never involute.  
\end{cor}
We need two following lemmas to get necessary and sufficient conditions for the integrability of $ \mathcal{D}\oplus sp\{\bar{U},\bar{V}\} $. 

\begin{lem}
	Let $ M $ be a semi-invariant submanifold. Then, we have 
	\begin{eqnarray*}
		\bar{g}(\mathbf{h}(K,L),\bar{G}Z)&=&\bar{g}(\nabla_{K}Z,\bar{G}L)\\
		\bar{g}(\mathbf{h}(K,L),\bar{H}Z)&=&\bar{g}(\nabla_{K}Z,\bar{H}L)\\
		\bar{g}(\mathbf{h}(K,L),\bar{J}Z)&=&\bar{g}(\nabla_{K}Z,\bar{J}L)
	\end{eqnarray*}
	for all vector fields $ K\in \Gamma(TM) $, $ L\in \Gamma(\mathcal{D}) $ and $ Z\in \Gamma(\mathcal{D}) $.
\end{lem}

\begin{proof}
	Let $ N=\bar{G}Z $ then from (\ref{g(h(x,y),n)}) we have 
	\begin{eqnarray*}
		\bar{g}(\mathbf{h}(K,L),\bar{G}Z)=\bar{g}(A_{\bar{G}Z}K,L)=-\bar{g}((\bar{\nabla}_{K}\bar{G})Z+\bar{G}\bar{\nabla}_{K}Z,L).
	\end{eqnarray*}
	On the other hand from (\ref{normallik (g(NablaXGY))}) we get 
	\begin{eqnarray*}
		\bar{g}(\mathbf{h}(K,L),\bar{G}Z)=\bar{g}(\bar{\nabla}_{K}Z,\bar{G}L).
	\end{eqnarray*}
	By following same steps, the equations:
	$\bar{g}(\mathbf{h}(K,L),\bar{H}Z)=\bar{g}(\nabla_{K}Z,\bar{H}L)$ and  $
	\bar{g}(\mathbf{h}(K,L),\bar{J}Z)=\bar{g}(\nabla_{K}Z,\bar{J}L)$ can be obtained.
\end{proof}
\begin{lem} \label{lemma2.4}
	For $ M $ we have $ [K,\bar{U}] $ and $[K,\bar{V}] \in \Gamma(\mathcal{D}\oplus sp\{\bar{U},\bar{V}\})$.
\end{lem}
\begin{proof}
	By using (\ref{gaussdenk}) and (\ref{altmanifoldnablaxU}) we have 
	\begin{equation*}
	\bar{g}([K,\bar{U}],L)=\bar{g}(\nabla_{\bar{U}}L,K)
	\end{equation*}
	for each $ L\in \Gamma(\mathcal{D}^{\bot}) $ and $ K\in \Gamma(\mathcal{D}) $.
	Now we take $ Z\in \Gamma(\mathcal{D})$ such that $ K=\bar{G}Z $ and from (\ref{g(h(x,y),n)})  we get 
	\begin{equation*}
	\bar{g}(\nabla_{\bar{U}}L,K)=\bar{g}(\mathbf{h}(\bar{U},Z),\bar{G}L)=0.
	\end{equation*}
	Thus $ \bar{g}([K,\bar{U}],L)=0 $ and by following same steps we get $ \bar{g}([K,\bar{V}],L)=0 $, it follows the assertion of the lemma. 
\end{proof}

\begin{thm}
	The distribution $ \mathcal{D}\oplus sp\{\bar{U},\bar{V}\} $ is involutive if and only if we have 
	\begin{equation} \label{d+spuv int}
	\mathbf{h}(K,\bar{G}L)=\mathbf{h}(\bar{G}K,L). 
	\end{equation}
\end{thm}
\begin{proof}
	From (\ref{lemma1.1.5}) we obtain 
	\begin{equation}  \label{teorem2.4ispat}
	\mathbf{h}(K,PL)-C\mathbf{h}(K,L)+Q\nabla_{K}L=0
	\end{equation}
	for all $ K,L \in \Gamma(\mathcal{D}) $. Since $ \mathbf{h} $ is symmetric substituting $ L  $ by $ K $ in (\ref{teorem2.4ispat}) we get $ \mathbf{h}(K,PL)-\mathbf{h}(L,PK) =Q[K,L]$. In this way $ [K,L] \in  \mathcal{D}\oplus sp\{\bar{U},\bar{V}\}$ if and only if (\ref{d+spuv int}) is satisfied. Taking into account (\ref{lemma2.4}), the proof is completed. 
\end{proof}
Finally we obtain a result for total umbilical semi-invariant submanifold. 
\begin{thm}
	 If $ M $ is a total umbilical submanifold then $ M $ is an invariant submanifold. 
\end{thm}
\begin{proof}
	Let $ M $ be a total umbilical semi-invariant submanifold. Then for $\forall Z \in \Gamma(\mathcal{D})^{\bot} $ from (\ref{totally umbl}) we have 
	\begin{equation*}
	\mathbf{h}(Z,\bar{U})=\bar{g}(Z,\bar{U})\mu=0.
	\end{equation*}
	On the other hand from (\ref{h(X,U)}) we have $ \mathbf{h}(Z,\bar{U})-QZ=\bar{G}Z$. Thus $ \bar{G}Z=0 $ and $ \mathcal{D}^{\bot}=0 $. So $M  $ is an invariant submanifold. 
\end{proof}
From above theorem we obtain following corollary. 
\begin{cor}
	There does not exist total umbilical proper semi-invariant submanifold  of a normal complex contact metric manifold. 
\end{cor}

\end{document}